\documentclass[smallextended]{amsart}
\usepackage[top=1in, bottom=1in, left=1.25in, right=1.25in]{geometry}
\usepackage{mathrsfs}
\usepackage{stackrel,amssymb}
\usepackage{amsfonts}
\usepackage[linkcolor=red,citecolor=blue]{hyperref}
\usepackage{MnSymbol}
\usepackage{latexsym}
\usepackage[all,cmtip]{xy}
\usepackage{amsmath}
\usepackage{color,ifsym}
\usepackage{graphicx}
\usepackage{fancybox}
\usepackage{longtable}
\usepackage{lscape}
\usepackage{float}
\usepackage{mathdots}
\usepackage{enumerate}
\usepackage{multirow}

\newtheorem{thm}{Theorem}
\newtheorem{lem}[thm]{Lemma}        
\newtheorem{cor}[thm]{Corollary}
\newtheorem{prop}[thm]{Proposition}

\newtheorem{mthm}[thm]{Main Theorem}

\newtheorem{conj}[thm]{Conjecture}

\newtheorem{exam}[thm]{Example}
\newtheorem{rem}{{\it Remark}} 

\theoremstyle{plain} 
\newcommand{\thistheoremname}{}
\newtheorem*{genericthm*}{\thistheoremname}
\newenvironment{namedthm*}[1]
{\renewcommand{\thistheoremname}{#1}%
	\begin{genericthm*}}
	{\end{genericthm*}}

\begin{document}
	\title[location of reducibility points]{location of reducibility points of induced representations I: A toy example}
   
	\author{CAIHUA LUO}
	\address{department of mathematics \\ bar-ilan university \\ ramat-gan, 5290000\\ israel}
	\email{chluo@amss.ac.cn}
	\date{}
\maketitle
\begin{abstract}
		By analyzing the singularity of standard intertwining operators, we provide a new way to understand the explicit location of reducibility points of induced representations of two Speh representations for general linear groups over a $p$-adic field. Through playing with this toy example, it seems that the analytic approach, in the spirit of M{\oe}glin--Waldspurger, could play a role for analogous reducibility problems in the setting of classical groups. On the other hand, it also stimulates the arising of some interesting questions.
\end{abstract}

\section{introduction}
  Let $G_n=GL_n(F)$ be the general linear group of rank $n$ defined over a $p$-adic field $F$. Given a unitary supercuspidal representation $\tau$ of $G_m$ and two positive integers $c,~a$, we can attach a Speh representation $\rho_c(\tau_a)$, and it would be interesting to see when the following parabolic induction representation
  \[\rho_c(\tau_a)|det(\cdot)|^s\times \rho_d(\tau_b)|det(\cdot)|^{-s} \]
  is reducible, here $s\in\mathbb{C}$ (Please refer to Section \ref{mainthms} for the notions). Indeed, by analyzing Langlands--Shahidi's normalized intertwining operators, M{\oe}glin--Waldspurger have proved the existence of a ``good'' bounded domain for the location of possible reducibility points, i.e., given by the so-called condition ``li{\'e}s'' in \cite[Lemma I.6.3]{moeglin1989residue}, it seems that was the best result by then. Other than that, one could not say more on the explicit location of reducibility points, even for the basic case $b=1$ or $d=1$ before Tadi{\'c} and Lapid--M{\'{\i}}nguez's recent work (see \cite{tadic2014irreducibility,lapidminguez2016Innerforms}). In the short note, we give a natural interpretation of the condition ``li{\'e}s'' in terms of normalization factors of intertwining operators, which in turn provides more information on the condition for our induced representation to be reducible (see Remark \ref{liescondi}). This is achieved by following closely M{\oe}glin--Waldspurger's argument in a subtle way. By doing so, it gives us a sense that the condition ``li{\'e}s'' might be an ``if and only if'' condition, which does hold provided that a strong version of Main Theorem \ref{mthm1}, i.e., Conjecture \ref{conjnonzero}, is established. Concerning our approach in the spirit of M{\oe}glin--Waldspurger, we can establish a special case of Conjecture \ref{conjnonzero}, see Proposition \ref{conjspecial}, via Cai--Friedberg--Ginzburg--Kaplan's local coefficient theory of $(k,~c)$-model, as opposed to the well-known Shahidi's local coefficient theory of Whittaker model, involved in the profound generalized doubling method (see \cite{cai2018doubling,cai2019doubling}). On the other hand, such an expectation is really true proved recently by Tadi{\'c} and Lapid--M{\'{\i}}nguez via an alternative way, i.e., a detailed analysis of Jacquet modules in \cite{tadic2014irreducibility,lapidminguez2016Innerforms} (see Theorem \ref{redthm}). Indeed, using mainly the Jacquet module tool in a combinatorial way, they have obtained some beautiful criterion of irreducibility for certain induced representations including the toy example investigated in the paper, and proposed some exciting conjectures in a series of papers (see \cite{lapidminguez2018Squareirred,lapidminguez2020Conjecture,lapid2018KazhdanLusztigpolynomial}). We hope our analytic approach could shed some light on the general case.
  
  The structure of the paper is as follows. In the next section, we will first state our Main Theorem \ref{mthm1} and then provide some applications, while its proof is given in the last section.

\section{main theorems}  \label{mainthms}
Let $F$ be a $p$-adic field with absolute value $|\cdot|$, $\mathfrak{w}$ be its uniformizer, $\mathcal{O}$ be its ring of integers, and $\mathbb{F}_q=\mathcal{O}/\mathfrak{w}\mathcal{O}$ be its residual field. For $n\in \mathbb{Z}_+$ the set of positive integers, denote $G_n=GL_n(F)$ to be the general linear group of rank $n$.

Let $\tau$ be a unitary supercuspidal representation of $G_t=GL_t(F)$ with $t|k$, and $\tau_a$ be the unique discrete series subrepresentation of the normalized parabolic induced representation \[\tau|det(\cdot)|^\frac{a-1}{2}\times\tau|det(\cdot)|^{\frac{a-3}{2}}\times\cdots \times \tau|det(\cdot)|^{-\frac{a-1}{2}}:=Ind^{G_k}(\tau|det(\cdot)|^{\frac{a-1}{2}}\otimes\cdots \otimes \tau|det(\cdot)|^{-\frac{a-1}{2}})\]
of $G_k=GL_k(F)$ with $a=\frac{k}{t}$. In particular, $\tau_1=\tau$. Attached to $\tau_a$ and $c\in \mathbb{Z}_+$, we denote $\rho_c(\tau_a)$ to be the associated Speh representation of $G_{kc}=GL_{kc}(F)$, i.e., the unique Langlands quotient of \[\tau_a|det(\cdot)|^{\frac{c-1}{2}}\times\tau_a|det(\cdot)|^{\frac{c-3}{2}}\times \cdots\times \tau_a|det(\cdot)|^{-\frac{c-1}{2}}\]
or the unique subrepresentation of \[\tau_a|det(\cdot)|^{-\frac{c-1}{2}}\times\tau_a|det(\cdot)|^{-\frac{c-3}{2}}\times \cdots\times \tau_a|det(\cdot)|^{\frac{c-1}{2}}.\]
For $\bar{\beta}=(\beta_1,~\beta_2)\in 
\mathbb{Z}_+^2$, let $P_{\bar{\beta}}$ be the standard maximal parabolic subgroup of $G_{\beta_1+\beta_2}$ with Levi factor $M_{\bar{\beta}}=G_{\beta_1}\times G_{\beta_2}$, $\sigma=(12)$ be the permutation such that 
\[\sigma M_{\bar{\beta}}=G_{\beta_2}\times G_{\beta_1}. \]
For irreducible admissible representations $\pi_i$ of $G_{\beta_i}$, $i=1,~2$, and $\bar{s}=(s_1,~s_2)\in \mathbb{C}^2$, we have the non-normalized standard intertwining operator (see \cite{moeglin1989residue})	
\[M(\sigma,~\pi_1\otimes \pi_2,~\bar{s}):~\pi_1|det(\cdot)|^{s_1}\times \pi_2|det(\cdot)|^{s_2}\longrightarrow \pi_2|det(\cdot)|^{s_2}\times \pi_1|det(\cdot)|^{s_1},\]
and its normalized form
\[N(\sigma,~\pi_1\otimes\pi_2,~\bar{s}):=\gamma(\sigma,~\pi_1\otimes\pi_2,~\bar{s})^{-1}M(\sigma,~\pi_1\otimes \pi_2,~\bar{s}), \]
where $$\gamma(\sigma,~\pi_1\otimes\pi_2,~\bar{s}):=L(s_1-s_2,~\pi_1\times \pi_2^\vee)\cdot L(1+s_1-s_2,~\pi_1\times \pi_2^\vee)^{-1}$$ is slightly different from the factor in \cite{moeglin1989residue}. They are the same up to scalars in $\mathbb{C}[q^{-s},~q^{s}]^\star$ the set of invertible elements in $\mathbb{C}[q^{-s},~q^{s}]$. Here $(\cdot)^\vee$ means taking the contragredient.

In the paper, we consider only the case $\pi_1=\rho_c(\tau_a)$ and $\pi_2=\rho_d(\tau_b)$ with $a,~b,~c,~d\in \mathbb{Z}_+$, and $\bar{s}=(s,~-s)\in \mathbb{C}^2$ for simplicity. The case $\rho_c(\tau_a)\otimes \rho_d(\tau'_b)$ with $\tau\not\simeq \tau'$, up to twisting by a unitary unramfied charater, could be discussed easily following from \cite{silberger1980special} and will be omitted.
Denote
\[\alpha(s,~\rho_c(\tau_a),~\rho_d(\tau_b)):=\prod_{j=\frac{|c-d|}{2}}^{\frac{c+d-2}{2}}L(2s-j,~\tau_a\times \tau^\vee_b),\quad \beta(s,~\rho_c(\tau_a),~\rho_d(\tau_b)):=\prod_{j=\frac{|c-d|}{2}}^{\frac{c+d-2}{2}}L(2s+j+1,~\tau_a\times\tau_b^\vee). \]
Then an easy calculation shows that
\[\gamma(\sigma,~\rho_c(\tau_a)\otimes \rho_d(\tau_b),~\bar{s})=\prod_{j=\frac{|c-d|}{2}}^{\frac{c+d-2}{2}}\frac{L(2s-j,~\tau_a\times \tau_b^\vee)}{L(2s+j+1,~\tau_a\times\tau_b^\vee)}=\frac{\alpha(s,~\rho_c(\tau_a),~\rho_d(\tau_b))}{\beta(s,~\rho_c(\tau_a),~\rho_d(\tau_b))}. \]

\begin{mthm} \label{mthm1}
	Keep the notions as above. We have	
	\[M^*(s,~\rho_c(\tau_a)\otimes \rho_d(\tau_b)):=\frac{1}{\alpha(s,~\rho_c(\tau_a),~\rho_d(\tau_b))}M(\sigma,~\rho_c(\tau_a)\otimes \rho_d(\tau_b),~\bar{s})\]
	is holomorphic for $s\in \mathbb{C}$.
\end{mthm}
With the help of \cite[Lemma I.5]{moeglin1989residue}, we can first obtain
\begin{cor}\label{sufficient}
	Maintain the notation as before. The reducibility points of the representation \[\rho_c(\tau_a)|det(\cdot)|^s\times \rho_d(\tau_b)|det(\cdot)|^{-s}\]
	are exactly the poles of
	\[\beta(s,~\rho_c(\tau_a),~\rho_d(\tau_b))\cdot\beta(-s,~\rho_c(\tau_a),~\rho_d(\tau_b)), \]
	provided that the normalization factors $\alpha(s,~\rho_c(\tau_a), ~\rho_d(\tau_b))^{-1}$ and $\beta(s,~\rho_c(\tau_a),~ \rho_d(\tau_b))^{-1}$ are co-prime in $\mathbb{C}[q^{-s},~q^s]$.
\end{cor}	
\begin{proof}
	Recall that \cite[Lemma I.5]{moeglin1989residue} says that the induced representation
	\[\rho_c(\tau_a)|det(\cdot)|^{s}\times \rho_d(\tau_b)|det(\cdot)|^{-s}\mbox{is irreducible at the point }s_0\]
	if and only if 
	\[N(\sigma,~\rho_c(\tau_a)\otimes\rho_d(\tau_b),~\bar{s})\mbox{ and } N(\sigma,~\rho_d(\tau_b)\otimes\rho_c(\tau_a),~-\bar{s}) \mbox{ are holomorphic at }s_0. \]
	Note that our Main Theorem \ref{mthm1} says that 
	\[\frac{1}{\alpha(s,~\rho_c(\tau_a),~\rho_d(\tau_b))}M(\sigma,~\rho_c(\tau_a)\otimes \rho_d(\tau_b),~\bar{s})
	\mbox{ is holomorphic for }s\in \mathbb{C}, \]
	thus the condition
	\[\alpha(s,~\rho_c(\tau_a), ~\rho_d(\tau_b))^{-1}\mbox{ and } \beta(s,~\rho_c(\tau_a),~ \rho_d(\tau_b))^{-1}\mbox{ are co-prime}\]
	implies that \[N(\sigma,~\rho_c(\tau_a)\otimes\rho_d(\tau_b),~\bar{s})=\frac{\beta(s,~\rho_c(\tau_a),~\rho_d(\tau_b))}{\alpha(s,~\rho_c(\tau_a),~\rho_d(\tau_b))}M(\sigma,~\rho_c(\tau_a)\otimes \rho_d(\tau_b),~\bar{s}) \] is holomorphic at the point $s_0$ (resp. -$s_0$) if and only if 
	\[\beta(s,~\rho_c(\tau_a),~\rho_d(\tau_b))\mbox{ is holomorphic at the point $s_0$ (resp. -$s_0$).} \]
	Whence finishing the proof.
\end{proof}	
Moreover, with an extra help of the preservation of reducibility property under the Zelevinsky--Aubert dual \cite{zelevinsky1980induced,aubert1995dualite}, i.e.,
\[\rho_c(\tau_a)|det(\cdot)|^s\times \rho_d(\tau_b)|det(\cdot)|^{-s}\mbox{\bf and }\rho_a(\tau_c)|det(\cdot)|^s\times \rho_b(\tau_d)|det(\cdot)|^{-s} \mbox{\bf share the same reducibility},\]
we can say a little bit more as follows.
\begin{cor}\label{sufficientAubert}
	Retain the notation as previous. The reducibility points of the representation \[\rho_c(\tau_a)|det(\cdot)|^s\times \rho_d(\tau_b)|det(\cdot)|^{-s}\]
	are exactly the poles of
	\[\beta(s,~\rho_c(\tau_a),~\rho_d(\tau_b))\cdot\beta(-s,~\rho_c(\tau_a),~\rho_d(\tau_b)), \]
	provided that the normalization factors $\alpha(s,~\rho_a(\tau_c), ~\rho_b(\tau_d))^{-1}$ and $\beta(s,~\rho_c(\tau_a),~ \rho_d(\tau_b))^{-1}$ are co-prime in $\mathbb{C}[q^{-s},~q^s]$.
\end{cor}
\begin{proof}
	As discussed earlier, the only point is to realize that the normalization factor $\beta(s,~\rho_c(\tau_a),~ \rho_d(\tau_b))$ is invariant with respect to the following symmetries
	\[(c,~d)\mapsto (d,~c),~(a,~b)\mapsto (b,~a), \mbox{ and }(c,~d)\mapsto (a,~b). \]
	Which follows easily from the calculation
	\[\beta(s,~\rho_c(\tau_a),~ \rho_d(\tau_b)):=\prod_{j=\frac{|c-d|}{2}}^{\frac{c+d-2}{2}}L(2s+j+1,~\tau_a\times\tau_b^\vee)=\prod_{j=\frac{|c-d|}{2}}^{\frac{c+d-2}{2}}\prod_{k=\frac{|a-b|}{2}}^{\frac{a+b-2}{2}}L(2s+j+k+1,~\tau\times\tau^\vee). \]
\end{proof}
Let us take a break to see what the conditions in our above corollaries say.
\begin{lem}\label{lemreducedCalcu}
	Use the same notation as defined earlier. We have  \[\alpha(s,~\rho_c(\tau_a), ~\rho_d(\tau_b))^{-1}~ (\mbox{resp. } \alpha(s,~\rho_a(\tau_c), ~\rho_b(\tau_d))^{-1})\mbox{ and } \beta(s,~\rho_c(\tau_a),~ \rho_d(\tau_b))^{-1} \mbox{ are co-prime}\]
	if and only if
	\[|c-d|\geq min\{a-1,~b-1\}~(\mbox{resp. }|a-b|\geq min\{c-1,~d-1\}). \]
\end{lem}
\begin{proof}
	This follows from an easy calculation as follows.
	\[\alpha(s,~\rho_c(\tau_a),~\rho_d(\tau_b)):=\prod_{j=\frac{|c-d|}{2}}^{\frac{c+d-2}{2}}L(2s-j,~\tau_a\times \tau_b^\vee)=\prod_{j=\frac{|c-d|}{2}}^{\frac{c+d-2}{2}}\prod_{k=\frac{|a-b|}{2}}^{\frac{a+b-2}{2}}L(2s-j+k,~\tau\times \tau^\vee), \]
	and
	\[\beta(s,~\rho_c(\tau_a),~ \rho_d(\tau_b)):=\prod_{j=\frac{|c-d|}{2}}^{\frac{c+d-2}{2}}L(2s+j+1,~\tau_a\times\tau_b^\vee)=\prod_{j=\frac{|c-d|}{2}}^{\frac{c+d-2}{2}}\prod_{k=\frac{|a-b|}{2}}^{\frac{a+b-2}{2}}L(2s+j+k+1,~\tau\times\tau^\vee). \]
	Thus the co-prime conditions are equivalent to saying that the minimal number in $\{j-k:~\cdots \}$ is bigger than the maximal number in $\{-j-k-1:~\cdots\}$, i.e.,
	\[\frac{|c-d|}{2}-\frac{a+b-2}{2}>-\frac{|c-d|}{2}-\frac{|a-b|}{2}-1 ~(\mbox{resp. }\frac{|a-b|}{2}-\frac{c+d-2}{2}>-\frac{|c-d|}{2}-\frac{|a-b|}{2}-1 ), \]
	that is,
	\[|c-d|\geq min\{a-1,~b-1\}~(\mbox{resp. }|a-b|\geq min\{c-1,~d-1\}). \]
\end{proof}	
Given the calculation in Lemma \ref{lemreducedCalcu}, we can easily see that two basic examples, of which the location of reducibility points seems unknown before Tadi{\'c} and Lapid--M{\'{\i}}nguez's recent work, lie in the setting of our Corollaries \ref{sufficient} and \ref{sufficientAubert} as follows.
\begin{exam}\label{exambasic}
	$\rho_c(\tau_a)|det(\cdot)|^s\times \rho_d(\tau)|det(\cdot)|^{-s}~ (i.e. ~b=1)$, and $\rho_c(\tau_a)|det(\cdot)|^s\times \tau_b|det(\cdot)|^{-s}~ (i.e.~d=1)$.	
\end{exam}
Now back to our discussion on the information of reducibility points we can extract from Main Theorem \ref{mthm1}. One may note that the conditions in Corollaries \ref{sufficient} and \ref{sufficientAubert} might be too strong, though they already imply something interesting, see Example \ref{exambasic}. Indeed, there does exist a weaker condition, arising from a simple observation, which guarantees the reducibility given in what follows. Write $$g.c.d\left(\alpha(s,~\rho_c(\tau_a),~\rho_d(\tau_b))^{-1},~\beta(s,~\rho_c(\tau_a),~\rho_d(\tau_b))^{-1} \right)$$ to be the greatest common divisor of the normalization factors $\alpha(s,\cdots)$ and $\beta(s,\cdots)$, and denote by $deg_{s_0}(\alpha(s,\cdots))$ (resp. $deg_{s_0}(\beta(s,\cdots))$) the order of pole of $\alpha(s,\cdots)$ (resp. $\beta(s,\cdots)$) at the point $s_0$. Then we have
\begin{cor}\label{corSuffi}
	Retain the notions as defined before. The induced representation \[\rho_c(\tau_a)|det(\cdot)|^s\times \rho_d(\tau_b)|det(\cdot)|^{-s}\]
	is reducible at those points $\pm s_0$, with $s_0<0$, characterized by one of the conditions
	\begin{enumerate}[(i).]
		\item $g.c.d\left(\alpha(s,~\rho_c(\tau_a),~\rho_d(\tau_b))^{-1},~\beta(s,~\rho_c(\tau_a),~\rho_d(\tau_b))^{-1} \right)|_{s=s_0}=1$, up to a non-zero scalar.
		\item $g.c.d\left(\alpha(s,~\rho_a(\tau_c),~\rho_b(\tau_d))^{-1},~\beta(s,~\rho_c(\tau_a),~\rho_d(\tau_b))^{-1} \right)|_{s=s_0}=1$, up to a non-zero scalar.	
		\item If (i) does not hold, $deg_{s_0}\left(\beta(s,~\rho_c(\tau_a),~\rho_d(\tau_b))\right)>deg_{s_0}\left(\alpha(s,~\rho_c(\tau_a),~\rho_d(\tau_b))\right)$.
		\item If (ii) does not hold, $deg_{s_0}\left(\beta(s,~\rho_c(\tau_a),~\rho_d(\tau_b))\right)>deg_{s_0}\left(\alpha(s,~\rho_a(\tau_c),~\rho_b(\tau_d))\right).$
	\end{enumerate}
\end{cor}	
\begin{proof}
	The former two conditions have been discussed in Corollaries \ref{sufficient} and \ref{sufficientAubert} respectively. For the latter parts, they are easy by-products of the following simple observation, (see \cite{waldspurger2003formule}),
	\[\mbox{\bf The order of zero of }M^*(s,~\rho_c(\tau_a)\otimes \rho_d(\tau_b))\mbox{ at }s_0\leq deg_{s_0}\left(\alpha(s,~\rho_c(\tau_a),~\rho_d(\tau_b))\right). \] 
\end{proof}
Let us take a look at two extremely bad cases, in terms of our conditions, to see what we can get from Corollary \ref{corSuffi}.
\begin{exam}\label{exambad1}
	$\rho_c(\tau_c)|det(\cdot)|^s\times \rho_c(\tau_c)|det(\cdot)|^{-s}~(i.e.~c=d=a=b)$: we have \[\alpha(s,\cdots)=\prod_{j_1=0}^{c-1}L(2s-j_1,~\tau_c\times \tau_c^\vee)=\prod_{j_1=0}^{c-1}\prod_{j_2=0}^{c-1}L(2s-j_1+j_2,~\tau\times \tau^\vee),\]
	and
	\[\beta(s,\cdots)=\prod_{j_1=0}^{c-1}L(2s+j_1+1,~\tau_c\times \tau_c^\vee)=\prod_{j_1=0}^{c-1}\prod_{j_2=0}^{c-1}L(2s+j_1+j_2+1,~\tau\times \tau^\vee). \]
	Thus the poles of $\alpha(s,\cdots)$ and $\beta(s,\cdots)$ counted with multiplicity are listed in terms of $2s$ by the following matrices respectively, red marked are their common poles, 
	\[ \left[\begin{array}{ccccc}
	{\color{red}-(c-1)} & {\color{red}\cdots}&{\color{red}\cdots} & {\color{red} -1} & 0 \\ 
	{\color{red}\vdots} & &{\color{red}\udots} & &\vdots\\
	{\color{red}\vdots} & {\color{red}\udots} & & & \vdots \\ 
	{\color{red}-1} & & &  & \vdots \\ 
	0 & \cdots &\cdots &\cdots  & (c-1)
	\end{array} \right]_{\alpha(s,\cdots),}\qquad 
	\left[\begin{array}{ccccc}
	-2(c-1)-1 &\cdots&\cdots & \cdots & -(c-1)-1 \\ 
	\vdots & & & &{\color{red} -(c-1)}\\
	\vdots &  & &{\color{red}\udots} & {\color{red}\vdots} \\ 
	\vdots & &{\color{red}\udots} &  & {\color{red}\vdots} \\ 
	-(c-1)-1 &	{\color{red}-(c-1)} &{\color{red}\cdots} &{\color{red}\cdots}  & {\color{red} -1}
	\end{array} \right]_{\beta(s,\cdots).} \]
	Therefor, by Corollary \ref{corSuffi}, we see that all poles of $\beta(s,~\rho_c(\tau_c),~\rho_c(\tau_c))\cdot\beta(-s,~\rho_c(\tau_c),~\rho_c(\tau_c))$ except the following points 
	\[2s\in\left\{\pm 1,~\pm 2,~\cdots,~\pm \left\lfloor\frac{c}{2} \right\rfloor\right\} \]
	are known to be reducible points of our induced representation.
\end{exam}	
\begin{exam}\label{exambad2}
	$\rho_c(\tau_a)|det(\cdot)|^s\times \rho_c(\tau_a)|det(\cdot)|^{-s}~(i.e.~c=d\neq  a=b)$: we have \[\alpha(s,~\rho_c(\tau_a),~\rho_c(\tau_a))=\prod_{j_1=0}^{c-1}L(2s-j_1,~\tau_a\times \tau_a^\vee)=\prod_{j_1=0}^{c-1}\prod_{j_2=0}^{a-1}L(2s-j_1+j_2,~\tau\times \tau^\vee),\]
	and 
	\[\beta(s,\cdots)=\prod_{j_1=0}^{c-1}L(2s+j_1+1,~\tau_a\times \tau_a^\vee)=\prod_{j_1=0}^{c-1}\prod_{j_2=0}^{a-1}L(2s+j_1+j_2+1,~\tau\times \tau^\vee). \]
	Following from taking the Zelevinsky--Aubert dual as in Corollary \ref{sufficientAubert}, we can assume $c>a$, then by Corollary \ref{corSuffi} and a similar calculation as done in Example \ref{exambad1}, we have that all poles of $\beta(s,~\rho_c(\tau_c),~\rho_c(\tau_c))\cdot\beta(-s,~\rho_c(\tau_c),~\rho_c(\tau_c))$ except the following points 
	\[2s\in \left\{\pm 1,~\pm 2,~\cdots,~\pm \left\lfloor\frac{a}{2} \right\rfloor\right\} \]
	are known to be reducible points of our induced representation.
\end{exam}
\begin{rem}\label{liescondi}
	Given our Main Theorem \ref{mthm1}, with the help of \cite[Lemma I.5]{moeglin1989residue} and the functional equation, up to an invertible scalar, 
	\[N(\sigma,~\rho_c(\tau_a)\otimes\rho_d(\tau_b),~\bar{s})\circ N(\sigma,~\rho_d(\tau_b)\otimes\rho_c(\tau_a),~-\bar{s})=id. \]
	one could readily see that the location of reducibility points is governed by the poles of 
	\[\beta(s,~\rho_c(\tau_a),~\rho_d(\tau_b))\cdot\beta(-s,~\rho_c(\tau_a),~\rho_d(\tau_b)), \]
	which is exactly the condition what M{\oe}glin--Waldspurger's notion of ``li{\'e}s'' says in \cite[Lemma I.6.3]{moeglin1989residue}. Moreover, our Main Theorem \ref{mthm1} could be proved in an another way without the profound Langlands--Shahidi theory utilized heavily in \cite{moeglin1989residue}, i.e., following from our new argument illustrated in \cite{luo2021casselmanshahidiconjecture}, and will be discussed as a necessary ingredient for analogous results for classical  groups in \cite{luo2021holomorphicityClassicalGroup}. In view of this, one could readily see that similar results hold for the general linear groups over division algebras if we could prove \cite[Lemma I.5 (i)]{moeglin1989residue} in this setting using the tool of Jacquet module directly, which we leave for a future work.	
\end{rem}
Based on the above analysis, especially the observation Remark \ref{liescondi}, we are curious to see if our approach in the spirit of M{\oe}glin--Waldspurger could reach the point that it answers the reducibility problem completely, even though the problem has been settled recently by Tadi{\'c} and Lapid--M{\'{\i}}nguez in \cite{tadic2014irreducibility,lapidminguez2016Innerforms} via a detailed analysis of Jacquet modules in a combinatorial way. Excited to say that, a conjectural strong version of Main Theorem \ref{mthm1} can achieve this in what follows. 
\begin{conj}\label{conjnonzero}
	Use the same notation as in Main Theorem \ref{mthm1}. We have
	\[M^*(s,~\rho_c(\tau_a)\otimes \rho_d(\tau_b)):=\frac{1}{\alpha(s,~\rho_c(\tau_a),~\rho_d(\tau_b))}M(\sigma,~\rho_c(\tau_a)\otimes \rho_d(\tau_b),~\bar{s}) \]
	is always non-zero for $s\in \mathbb{C}$, i.e., a non-zero intertwining operator.
\end{conj}
\begin{thm}[Tadi{\'c}, Lapid--M{\'{\i}}nguez \cite{tadic2014irreducibility,lapidminguez2016Innerforms}]\label{redthm}
	Corollary \ref{sufficient} holds unconditionally.
\end{thm}  
\begin{proof}
	 As seen in Remark \ref{liescondi}, reducibility points are govern by the poles of 
	 \[\beta(s,~\rho_c(\tau_a),~\rho_d(\tau_b))\cdot\beta(-s,~\rho_c(\tau_a),~\rho_d(\tau_b)). \]
	 Hence it suffices to show the converse also holds, i.e., poles of $\beta(s,~\cdots)\beta(-s,~\cdots)$ are controlled by the reducibility points of
	 \[\rho_c(\tau_a)|det(\cdot)|^s\times \rho_d(\tau_b)|det(\cdot)|^{-s}. \]
	 This follows easily from Conjecture \ref{conjnonzero} via the following functional equation, up to invertible elements,
	 \[M^*(s,~\rho_c(\tau_a)\otimes \rho_d(\tau_b))\circ M^*(-s,~ \rho_d(\tau_b)\otimes\rho_c(\tau_a))=\frac{1}{\beta(s,~\rho_c(\tau_a),~\rho_d(\tau_b))\cdot\beta(-s,~\rho_c(\tau_a),~\rho_d(\tau_b))}. \]
	 To be precise, if the right hand side is zero at some point $s_0$, then the left hand side is also zero at $s_0$. But Conjecture \ref{conjnonzero} says that $M^*(\pm s_0,~\cdots)\neq 0$, so the induced representation must be reducible at $s_0$. Whence Theorem \ref{redthm} holds. 
\end{proof}
Regarding Conjecture \ref{conjnonzero}, via Cai--Friedberg--Ginzburg--Kaplan's local coefficient theory of $(k,~c)$-model in \cite{cai2018doubling,cai2019doubling}, we have
\begin{prop}\label{conjspecial}
	Keep the notions as above. Conjecture \ref{conjnonzero} holds for the case $c=d$.
\end{prop}
\begin{proof}
	Recall that Cai--Friedberg--Ginzburg--Kaplan's local coefficient theory of $(k,~c)$-model, given by the following diagram, please refer to \cite{cai2018doubling,cai2019doubling} for the notation and results in details,
	\[ \xymatrix{ \rho_c(\tau_a)|det(\cdot)|^s\times \rho_c(\tau_b)|det(\cdot)|^{-s}\quad\ar[rr]^{M(\sigma,\rho_c(\tau_a)\otimes\rho_c(\tau_b),\bar{s})} \ar[dr]_{\lambda(s,\cdots)} & & \quad\rho_c(\tau_b)|det(\cdot)|^{-s}\times\rho_c(\tau_a)|det(\cdot)|^s \ar[dl]^{\lambda(-s,\cdots)} \\
		&	\mathbb{C}_{\psi} &
	}\]
	says that there exists a rational coefficient $C_\psi(s,~\cdots)$ given by, up to an invertible element in $\mathbb{C}[q^{-s},q^s]$,
	\[C_\psi(s,~\cdots)=\frac{\beta(-s,~\rho_c(\tau_a),~\rho_c(\tau_b))}{\alpha(s,~\rho_c(\tau_a),~\rho_c(\tau_b))} \]
	such that, up to an invertible element in $\mathbb{C}[q^{-s},q^s]$,
	\[\lambda(s,~\cdots)=C_\psi(s,~\cdots)\lambda(-s,~\cdots)\circ M(\sigma,~\rho_c(\tau_a)\otimes\rho_c(\tau_b),~\bar{s})=\beta(-s,~\cdots)\lambda(-s,~\cdots)\circ M^*(s,~\rho_c(\tau_a)\otimes \rho_c(\tau_b)). \]
	Note that $\lambda(s,~\cdots)$ is holomorphic in $s$ and non-zero (see \cite[P. 15]{cai2019doubling}), thus we obtain
	\[M^*(s,~\rho_c(\tau_a)\otimes \rho_c(\tau_b))\neq 0\mbox{ for }Re(s)\leq 0, \]
	which follows from the fact that $\beta(-s,~\cdots)$ has no poles at $Re(s)\leq 0$.
	
	On the other hand, the top diagram in Page 615 in \cite[I.5 Lemma]{moeglin1989residue} says that \[N(\sigma,~\rho_c(\tau_a)\otimes \rho_d(\tau_b),~\bar{s})=\frac{\beta(s,~\rho_c(\tau_a),~\rho_d(\tau_b))}{\alpha(s,~\rho_c(\tau_a),~\rho_d(\tau_b))}M(\sigma,~\rho_c(\tau_a)\otimes \rho_d(\tau_b),~\bar{s})\]
	is holomorphic and nonzero for $Re(s)\geq 0$. Note that \[\beta(s,~\rho_c(\tau_a),~\rho_d(\tau_b)):=\prod_{j=\frac{|c-d|}{2}}^{\frac{c+d-2}{2}}L(2s+j+1,~\tau_a\times\tau_b^\vee)\] 
	has no poles for $Re(s)\geq 0$, thus we know
	\[M^*(s,~\rho_c(\tau_a)\otimes \rho_d(\tau_b)):=\frac{1}{\alpha(s,~\rho_c(\tau_a),~\rho_d(\tau_b))}M(\sigma,~\rho_c(\tau_a)\otimes \rho_d(\tau_b),~\bar{s}) \]
	is always non-zero for $Re(s)\geq 0$.
	Whence $M^*(s,~\rho_c(\tau_a)\otimes \rho_c(\tau_b))$ is always non-zero for $s\in \mathbb{C}$.
\end{proof}
\begin{rem}\label{localeffic}
	One may notice that the local coefficient theory of $(k,~c)$-model for general linear groups $GL$ has not been written down in detail in \cite{cai2018doubling}. But one also sees that the unramified calculation in some sense has been carried out in \cite[P.63 equation (6.57)]{cai2018doubling} and the multiplicative property is quite standard as in \cite[Proposition 3.2.1]{shahidi1981certain}, thus the formula used in the above proof, i.e.,
	\[C_\psi(s,~\cdots)=\frac{\beta(-s,~\rho_c(\tau_a),~\rho_c(\tau_b))}{\alpha(s,~\rho_c(\tau_a),~\rho_c(\tau_b))} \]
	follows easily from the fact that $C_\psi(s,~\cdots)=1$ globally, see \cite[P.65]{cai2018doubling}, via a local-global argument. As this is not so related with our focus in the paper, we do not include the details herein.
\end{rem}

\section{proof of main theorem \ref{mthm1}}
In this section, we will follow M{\oe}glin--Waldspurger's work \cite{moeglin1989residue} to prove our Main Theorem \ref{mthm1} for $\tau_a$ and $\tau_b$ supercuspidal, and discrete series gradually. Please refer to \cite[Section IA]{moeglin1989residue} for the details.

\begin{proof}[Proof of Main Theorem \ref{mthm1} ($\tau_a$ and $\tau_b$ supercuspidal)]
	Note that \cite[Proposition I.10]{moeglin1989residue} and \cite[Lemma I.6.3]{moeglin1989residue} say that $N(\sigma,~\rho_c(\tau)\otimes \rho_d(\tau),~\bar{s})$ is holomorphic and non-zero for $2Re(s)>-1$, with possible poles at the reducible points of $\rho_c(\tau)|det(\cdot)|^{s}\times \rho_d(\tau)|det(\cdot)|^{-s}$, i.e., 
	\[2Re(s)\in \left[-\frac{c+d}{2},~-\frac{|c-d|}{2}\right)\cap \mathbb{Z}.\]
	Thus it is equivalent to showing that, via a simple analysis of the normalization factors $\alpha(s,~\rho_c(\tau),~\rho_d(\tau))$ and $\beta(s,~\rho_c(\tau),~\rho_d(\tau))$,
	\[M(\sigma,~\rho_c(\tau)\otimes \rho_d(\tau),~\bar{s})\mbox{ is holomorphic for }2Re(s)< -\frac{|c-d|}{2}. \tag{$\star$} \] 
	Which is equivalent to showing that 
	\[N(\sigma,~\rho_c(\tau)\otimes \rho_d(\tau),~\bar{s})\mbox{ has only simple poles possibly at }2Re(s)\in \left[-\frac{c+d}{2},~-\frac{|c-d|}{2}\right)\cap \mathbb{Z}.\tag*{(C1)} \]
	Viewing $\rho_d(\tau)$ as a subrepresentation of $\tau|det(\cdot)|^{-\frac{d-1}{2}}\times\tau|det(\cdot)|^{-\frac{d-3}{2}}\times \cdots\times\tau|det(\cdot)|^{\frac{d-1}{2}}$, we have $$N(\sigma,~\rho_c(\tau)\otimes \rho_d(\tau),~\bar{s})=\prod_{j_2=-\frac{d-1}{2}}^{\frac{d-1}{2}}N(\sigma,~\rho_c(\tau)\otimes \tau|det(\cdot)|^{j_2},~\bar{s}) \mbox{ (from large to small)}.$$
	For each $N(\sigma,~\rho_c(\tau)\otimes \tau|det(\cdot)|^{j_2},~\bar{s})$, applying \cite[Proposition I.10]{moeglin1989residue} and \cite[Lemma I.6.3]{moeglin1989residue} again, we know that it has possible poles only at 
	\[2Re(s)=-\frac{c+1}{2}+j_2\in \mathbb{Z}.\] 
	On one hand, $$M(\sigma,~\rho_c(\tau)\otimes \tau|det(\cdot)|^{j_2},~\bar{s})=\prod_{j_1=-\frac{c-1}{2}}^{\frac{c-1}{2}}M(\sigma,~\tau|det(\cdot)|^{j_1}\otimes \tau|det(\cdot)|^{j_2},~\bar{s}) \mbox{ (from small to large)}$$ via a reduced decomposition, and $M(\sigma,~\tau|det(\cdot)|^{j_1}\otimes \tau|det(\cdot)|^{j_2},~\bar{s})$ has only a simple pole possibly at (see \cite{bernstein1977induced})
	\[2Re(s)=j_2-j_1.\]
	On the other hand, an easy calculation of the normalization factor gives
	\[\gamma(\sigma,~\rho_c(\tau)\otimes \tau|det(\cdot)|^{j_2},~\bar{s})=\frac{L(2s-1-j_2-\frac{c-1}{2},~\tau\times \tau^\vee)}{L(2s-1-j_2+\frac{c-1}{2}+1,~\tau\times \tau^\vee)} \]
	Thus the inequality $2Re(s)=-\frac{c-1}{2}-1+j_2<j_2-j_1$ for any $j_1$ implies that
	\[N(\sigma,~\rho_c(\tau)\otimes \tau|det(\cdot)|^{j_2},~\bar{s})\mbox{ has only a simple pole possibly at } 2Re(s)=-\frac{c+1}{2}+j_2.  \] 
	Therefore, the inequality $-\frac{c+1}{2}+j_2\neq -\frac{c+1}{2}+j'_2$ if $j_2\neq j'_2$ implies that
	\[N(\sigma,~\rho_c(\tau)\otimes \rho_d(\tau),~\bar{s})\mbox{ has only simple poles possibly at }2Re(s)\in \left[-\frac{c+d}{2},~-\frac{|c-d|}{2}\right)\cap \mathbb{Z}. \]
	Whence finishing the Claim (C1).
\end{proof}

\begin{proof}[Proof of Main Theorem \ref{mthm1} ($\tau_a$ and $\tau_b$ discrete series)]
	Recall that $\tau_b$ is the unique subrepresentation of $\tau|det(\cdot)|^\frac{b-1}{2}\times \cdots\times \tau|det(\cdot)|^{-\frac{b-1}{2}}$. The key observation for our argument in what follows is that
	\[\rho_c(\tau_b)\simeq S(\rho_c(\tau)|det(\cdot)|^{\frac{b-1}{2}}\times\cdots\times\rho_c(\tau)|det(\cdot)|^{-\frac{b-1}{2}})\mbox{ (the unique subrepresentation)}. \]
	This follows from the fact that they share the same Zelevinsky--Aubert dual (see \cite{moeglinwaldspurger1986involution}). In view of this observation, viewing $\rho_c(\tau_b)$ as a subrepresentation of $$\rho_c(\tau)|det(\cdot)|^{\frac{b-1}{2}}\times\cdots\times\rho_c(\tau)|det(\cdot)|^{-\frac{b-1}{2}},$$
	we have the decomposition
	\[N(\sigma,~\rho_c(\tau_a)\otimes\rho_d(\tau_b),~\bar{s})=\prod_{j_2=-\frac{b-1}{2}}^{\frac{b-1}{2}}N(\sigma,~\rho_c(\tau_a)\otimes\rho_d(\tau)|det(\cdot)|^{j_2},~\bar{s})\mbox{ (from large to small)}.  \]
	For each $N(\sigma,~\rho_c(\tau_a)\otimes\rho_d(\tau)|det(\cdot)|^{j_2},~\bar{s})$, \cite[Proposition I.10]{moeglin1989residue} and \cite[Lemma I.6.3]{moeglin1989residue} imply that it has possible poles at
	\[2Re(s)\in \left[j_2-\frac{a-1}{2}-\frac{c+d}{2},~j_2-\frac{a-1}{2}-\frac{|c-d|}{2}\right)\cap \mathbb{Z}. \]
	On the other hand, 
	\[M(\sigma,~\rho_c(\tau_a)\otimes\rho_d(\tau)|det(\cdot)|^{j_2},~\bar{s})=\prod_{j_1=-\frac{a-1}{2}}^{\frac{a-1}{2}}M(\sigma,~\rho_c(\tau)|det(\cdot)|^{j_1}\otimes\rho_d(\tau)|det(\cdot)|^{j_2},~\bar{s}) \mbox{ (from small to large)}. \]
	Note that, by ($\star$), we know that $M(\sigma,~\rho_c(\tau)|det(\cdot)|^{j_1}\otimes\rho_d(\tau)|det(\cdot)|^{j_2},~\bar{s})$ has possible poles only at
	\[2Re(s)+j_1-j_2\geq -\frac{|c-d|}{2},~i.e.,~2Re(s)\in \left[j_2-j_1-\frac{|c-d|}{2},+\infty\right)\cap \mathbb{Z}. \]
	Thus the inequality $j_2-\frac{a-1}{2}\leq j_2-j_1$ says that for any $j_1$,
	\[j_2-\frac{a-1}{2}-\frac{|c-d|}{2}\leq j_2-j_1-\frac{|c-d|}{2}. \]
	Which in turn implies that, for $N(\sigma,~\rho_c(\tau_a)\otimes\rho_d(\tau)|det(\cdot)|^{j_2},~\bar{s})$,
	\[\mbox{all of its possible poles come from the normalization factor }\gamma(\sigma,~\rho_c(\tau_a)\otimes\rho_d(\tau)|det(\cdot)|^{j_2},~\bar{s}).\tag{$\star\star$}\]
	Therefore our Main Theorem \ref{mthm1} holds for the inducing data $\rho_c(\tau_a)\otimes\rho_d(\tau)$, as well as $\rho_c(\tau)\otimes \rho_d(\tau_b)$ which could be proved similarly and the detail will be left to the reader.
	
	Back to the general case, by an easy calculation, for $a\geq b$, we know that the normalization factors  $\alpha(\cdots)$ match on both sides of the decomposition
	\[N(\sigma,~\rho_c(\tau_a)\otimes \rho_d(\tau_b),~\bar{s})=\prod_{j_2=-\frac{b-1}{2}}^{\frac{b-1}{2}}N(\sigma,~\rho_c(\tau_a)\otimes \rho_d(\tau)|det(\cdot)|^{j_2},~\bar{s}) \mbox{ (from large to small)}.\]
	The case $a\leq b$ can be analyzed similarly as above by decomposing with respect to $\rho_c(\tau_a)$ instead of $\rho_d(\tau_b)$, and will be left as an exercise for the interested readers. Whence \[M^*(s,~\rho_c(\tau_a)\otimes \rho_d(\tau_b)):=\frac{1}{\alpha(s,~\rho_c(\tau_a),~\rho_d(\tau_b))}M(\sigma,~\rho_c(\tau_a)\otimes \rho_d(\tau_b),~\bar{s})\]
	is holomorphic for $s\in \mathbb{C}$, completing the proof of Main Theorem \ref{mthm1}.
\end{proof}

\paragraph*{\textbf{Acknowledgments}} The author would like to thank Eyal Kaplan for his kindness and help. Thanks are also due to the referee for his/her detailed comments. The author was supported by the ISRAEL SCIENCE FOUNDATION Grant Number 376/21.
			
\bibliographystyle{amsalpha}
\bibliography{ref}

			
\end{document}